\documentclass[a4paper,12pt]{amsart}

\usepackage{amssymb}
\usepackage[utf8]{inputenc}
\usepackage[T1]{fontenc}
\usepackage{textcomp}
\usepackage{tikz}
\usepackage{enumitem}

\usepackage{hyperref}
\usepackage{xcolor}
\hypersetup{%
    colorlinks=false,
    pdfborderstyle={/S/U/W 1}
}

\setenumerate[1]{label=(\roman{*}), ref=(\roman{*})}
\setenumerate[2]{label=(\alph{*}), ref=(\alph{*})}

\DeclareMathOperator\supp{supp}
\DeclareMathOperator\sgn{sgn}
\DeclareMathOperator\ext{ext}

\newcommand\N{\mathbb N}

\newcommand{\eps}{\varepsilon}

\newcommand{\ha}{h_{\mathcal{A},1}}
\newcommand{\hap}{h_{\mathcal{A},p}}
\newcommand{\A}{\mathcal{A}}
\newcommand{\mx}{M(x)}
\newcommand{\minfx}{M^\infty(x)}

\newtheorem{thm}{Theorem}[section]
\newtheorem{prop}[thm]{Proposition}
\newtheorem{lem}[thm]{Lemma}
\newtheorem{cor}[thm]{Corollary}

\theoremstyle{definition}
\newtheorem{defn}[thm]{Definition}

\newtheorem{example}[thm]{Example}

\newtheorem{quest}[thm]{Question}

\theoremstyle{remark}

\newtheorem{rem}[thm]{Remark}

\title[Delta-points in $\hap$ spaces]{Delta-points in
  Banach spaces generated by adequate families}

\author[T.~A.~Abrahamsen]{Trond A.~Abrahamsen}
\address[T.~A.~Abrahamsen]{Department of Mathematics, University of
    Agder, Postboks 422, 4604 Kristiansand, Norway.}
\email{trond.a.abrahamsen@uia.no}
\urladdr{http://home.uia.no/trondaa/index.php3}

\author[V.~Lima]{Vegard Lima}
\address[V.~Lima]{Department of Engineering Sciences, University of Agder,
    Postboks 509, 4898 Grimstad, Norway.}
\email{Vegard.Lima@uia.no}
\thanks{V.~Lima acknowledges the support by the Estonian Research Council grant PRG877}
\author[A.~Martiny]{Andr\'e Martiny}
\address[A.~Martiny]{Department of Mathematics, University of
    Agder, Postboks 422, 4604 Kristiansand, Norway.}
\email{andre.martiny@uia.no}



\begin{document}
    
\begin{abstract}
  We study delta-points in Banach spaces $\hap$ generated by adequate
  families $\mathcal A$ where $1 \le p < \infty$. When the
  familiy $\mathcal A$ is regular and $p=1,$ these spaces are known as
  combinatorial Banach spaces. When $p > 1$ we prove that neither
  $\hap$ nor its dual contain delta-points. Under the extra assumption that
  $\mathcal A$ is regular, we prove that the same is true when
  $p=1.$ In particular the Schreier spaces and their duals fail to have
  delta-points. If $\mathcal A$ consists of finite sets only we are able to rule out
  the existence of delta-points in $\ha$ and Daugavet-points in its dual.

  We also show that if $\ha$ is polyhedral,
  then it is either (I)-polyhedral or (V)-polyhedral
  (in the sense of Fonf and Vesel\'y).
\end{abstract}

\maketitle

\section{Introduction}
According  to Talagrand \cite{MR554378,MR736065},
a family $\mathcal{A}$ of subsets  of $\N$
is \emph{adequate} if $\mathcal{A}$ contains the singletons,
is  hereditary and is  compact with respect to the topology
of pointwise convergence.
We study the spaces $\hap$, $1\leq p < \infty$,
generated by an adequate family $\mathcal{A}$ by completing
the space $c_{00}$ of all finitely supported sequences with the norm
\begin{equation*}
  \|\sum_{i=1}^n a_i e_i\|
  = \sup_{A \in \mathcal{A}} (\sum_{i\in A} |a_i|^p)^{1/p}.
\end{equation*}
Examples of $\hap$ spaces are
$c_0$, $\ell_p$, $\ell_1(c_0)$ and the (higher order) Schreier space(s).
We study the presence or absence
of \emph{Daugavet}- and \emph{delta}-points
in $\hap$ spaces and their duals.

For a Banach space $X$ we denote by $B_X$, $S_X$ and $X^*$
the closed unit ball, the unit sphere and the topological dual of $X$
respectively.
Recall that a \emph{slice} of $B_X$ is a set
\begin{equation*}
  S(x^*, \delta) = \left\{x\in B_X \ : \ x^*(x) > 1- \delta \right\},
\end{equation*}
where $x^*\in S_{X^*}$ and $\delta > 0$.
Following \cite{AHLP} we say that $x \in S_X$ is a Daugavet-point
(resp.\ delta-point) if every element in the unit ball
(resp.\ $x$ itself) is in the closed convex hull of unit ball
elements that are almost at distance $2$ from $x$.
In \cite[Lemmas~2.2 and 2.3]{AHLP}, we find the following
characterization of Daugavet- and delta-points which
will serve as our definitions.
\begin{lem}\label{lem:Delta_point_crit}
  Let $X$ be a Banach space.
  Then $x \in S_X$ is a
  \begin{enumerate}
  \item \emph{delta-point} if and only if
    for every slice $S(x^*,\delta)$ of $B_X$ with $x \in S(x^*,\delta)$
    and for every $\varepsilon > 0$ there exists $y \in S(x^*,\delta)$
    such that $\|x - y\| \geq 2 - \varepsilon$.
  \item \emph{Daugavet-point} if and only if
    for every slice $S(x^*,\delta)$ of $B_X$
    and for every $\varepsilon > 0$ there exists $y \in S(x^*,\delta)$
    such that $\|x - y\| \geq 2 - \varepsilon$.
  \end{enumerate}
\end{lem}
A Banach space has the Daugavet property if (and only if)
every $x \in S_X$ is a Daugavet-point.
For more on Daugavet- and delta-points see
\cite{AHLP, abrahamsen2020daugavet, haller2020daugavet,
  jung2020daugavet, zoca2020lorthogonality}.

It is known that if a Banach space $X$ has the Daugavet property,
then $X$ does not have an unconditional basis
\cite[Corollary~2.3]{Kadets}.
In \cite[Theorem~4.7]{abrahamsen2020daugavet} it was shown
that there exists an $\ha$ space
(with 1-unconditional basis) with ``lots'' of Daugavet-points,
in the sense that its unit ball is the weak closure of its
Daugavet-points.
After some preliminaries in Section~\ref{sec:preliminaries},
we show in Section~\ref{sec:case-1-}
that if $1 < p < \infty$ then
neither $\hap$ nor $\hap^*$ have delta-points.

The case $p = 1$ studied in Section~\ref{sec:case-p-er-lik-1}
is not so clear cut.
As mentioned above there exist $\ha$ spaces with delta-points, but
if $\mathcal{A}$ consists of all subsets of $\N$, then
$\ha = \ell_1$ and by \cite[Theorem~2.17]{abrahamsen2020daugavet}
$\ell_1$ does not have delta-points.
We show that if the adequate family $\mathcal{A}$ contains
an infinite set then $\ha^*$ always has delta-points.
Our main focus in Section~\ref{sec:case-p-er-lik-1}
is on adequate families of finite sets.
If $\mathcal{A}$ is an adequate family of finite subsets
it follows from the results of \cite{abrahamsen2020daugavet}
that $\ha$ does not have delta-points.
If $\mathcal{A}$ is in addition spreading (see Section~\ref{sec:preliminaries})
then the $\ha$ spaces we get are the combinatorial Banach spaces
studied by Antunes, Beanland and Chu in \cite{le2019geometry}.
Using results from \cite{le2019geometry}
we show that if $\ha$ is a combinatorial Banach space,
then also $\ha^*$ fail to have delta-points.

In Section~\ref{sec:case-p-er-lik-1} we also show that
all extreme points of $\ha^*$ are $w^*$-exposed for all
families $\mathcal{A}$.
If $\mathcal{A}$ consists of finite sets only,
then the extreme points in $\ha^*$
are in fact $w^*$-strongly exposed.
We use this observation to show that such $\ha^*$
spaces fail to have Daugavet-points.

Finally, in Section~\ref{sec:polyhedrality}
we study polyhedrality in $\ha$ spaces.
A Banach space is \emph{polyhedral} if the unit ball of
every finite-dimensional subspace of $X$ is a polytope
\cite{KleePolyhedralSections}.
In \cite{le2019geometry} it was shown that combinatorial
Banach spaces are all (V)-polyhedral
(in the sense of Fonf and Veselý \cite{MR2057283}).
In fact, their proofs also work without the spreading property
and show that all $\ha$ spaces where the adequate family consists of finite
subsets of $\N$ are (V)-polyhedral.
We observe that $\ha$ is polyhedral if and only if $\ha$
is (V)-polyhedral if and only if $\mathcal{A}$ is an
adequate family of finite sets.
Furthermore, we show that $\ha$ is (I)-polyhedral
if and only if $\ha$ is (IV)-polyhedral if and only if
$\left\{A\in \mathcal{A} : i\in A \right\}$ is
finite for all $i\in \N$.

We conclude the paper with some questions arising from
the present work.

\section{Preliminaries}
\label{sec:preliminaries}
Recall that a Schauder basis $(e_i)_{i\in \N}$ of a Banach space $X$
is a \emph{1-unconditional basis} if for all $N \in \N$
and all scalars $a_1,\ldots,a_N$, $b_1,\ldots,b_N$ such that
$|a_i| \le |b_i|$ for $i=1,\ldots,N$ then the following
inequality holds:
\begin{equation*}
  \left\|\sum_{i=1}^N a_i e_i \right\|
  \le
  \left\|\sum_{i=1}^N b_i e_i \right\|.
\end{equation*}
A basis $(e_i)_{i \in \N}$ is \emph{normalized} if $\|e_i\| = 1$
for all $i$ and it is \emph{shrinking} if the biorthogonal
functionals $(e_i^*)_{i \in \N}$ is a basis for $X^*$.
For $x \in X$ the \emph{support} of $x$ is defined
by $\supp(x) = \{ i \in \N : e_i^*(x) \neq 0 \}$.

If $(e_i)_{i \in \N}$ is 1-unconditional, then by the classic
result of James it is shrinking
if and only if $X$ does not contain an isomorphic copy of $\ell_1$.
If $(e_i)_{i \in \N}$ is a 1-unconditional basis
then for any $A \subset \N$ the projection $P_A$ defined by
\begin{equation*}
  P_A(\sum_{i \in \N} x_i e_i) = \sum_{i \in A} x_i e_i
\end{equation*}
satisfies $\|P_A\| \le 1$.

As mentioned earlier we will study the existence of delta- and Daugavet-points in
sequence spaces with 1-unconditional bases
generated by adequate families of subsets of $\N$.

\begin{defn}\label{defn:adequatefamilies}
  A family $\mathcal{A}$ of subsets of $\N$ is \emph{adequate} if
  \begin{enumerate}
  \item
    $\mathcal{A}$ contains the empty set and the singletons:
    $\{i\} \in \mathcal{A}$ for all $i \in \N$.
  \item
    $\mathcal{A}$ is hereditary:
    If $A \in \mathcal{A}$ and
    $B \subseteq A$, then
    $B \in \mathcal{A}$.
  \item
    $\mathcal{A}$ is compact with respect to the topology
    of pointwise convergence:
    Given $A \subset \N$,
    if every finite subset of $A$ is in $\mathcal{A}$,
    then $A \in \mathcal{A}$.
  \end{enumerate}
\end{defn}
We denote by $\mathcal{A}^{\text{MAX}}$ the maximal elements of
$\mathcal{A}$, that is $A \in \mathcal{A}^{\text{MAX}}$
if $A \in \mathcal{A}$ and $B \in \mathcal{A}$ with
$A \subseteq B$ implies that $A = B$.

Let $c_{00}$ be the vector space of all finitely supported
sequences with standard basis $(e_i)_{i \in \N}$.
If $\mathcal{A}$ is adequate and $1 \le p < \infty$,
then $\hap$ is the completion of $c_{00}$ with respect to the norm
\begin{equation*}
  \left\|\sum_{i=1}^n a_i e_i\right\|
  = \sup_{A \in \mathcal{A}} \biggl(\sum_{i\in A} |a_i|^p\biggr)^{1/p}.
\end{equation*}
It is clear that $(e_i)_{i \in \N}$ is a normalized
1-unconditional basis for $\hap$.

If $\mathcal{A}$ is an adequate family of finite sets of $\N$
which is \emph{spreading}, that is, if $\{ k_1,\ldots,k_n \}\in \mathcal{A}$
and $k_i \le l_i$, then $\{l_1, \ldots, l_n \} \in \mathcal{A}$,
then $\mathcal{A}$ is often called a \emph{regular} family
of subsets of $\N$ and the space $\ha$ is called a combinatorial
Banach space and $\hap$ is its $p$-convexification
(see e.g. \cite{le2019geometry}).

From \cite{abrahamsen2020daugavet} we need the notion of
\emph{minimal norming subsets} for vectors in a Banach space
with 1-unconditional basis (see also the notion of
\emph{1-sets} in \cite{BeanlandExtremeCombinatorial}).

\begin{defn}
  \label{defn:cx}
  For any Banach space $X$ with 1-unconditional basis
  $(e_i)_{i \in \N}$ and for $x \in X$, define
  \[
    \mx := \left\{
      A\subseteq \N :
      \left\|P_Ax \right\| = \left\| x \right\|,
      \left\|P_Ax-x_ie_i\right\| < \left\| x\right\|,
      \ \mbox{for all}\ i\in A
    \right\},
  \]
  and
  \[
    \minfx := \left\{A\in \mx : |A|=\infty \right\}.
  \]
\end{defn}

For any $B = (b_i)_{i \in \N} \subseteq \N$,
assume that the elements are ordered, that is,
$b_i < b_{i+1}$ for all $i$.
For $n \in \N$ we define $B(n) := (b_i)_{i=1}^n$.
Let $X$ be a Banach space with a 1-unconditional basis
$(e_i)_{i \in \N}$, if $n \in \N$, then by 
\cite[Lemma~2.8]{abrahamsen2020daugavet}
the set $\bigcup_{D\in \minfx} \left\{D(n) \right\}$
is finite whenever $x \in S_X$. This means that there
exists $s = s(n) \in \N$ and sets $D_1,\ldots,D_s$ in $\minfx$
such that
\[
  \bigcup_{D\in \minfx} \left\{D(n) \right\}
  = \left\{D_1(n), D_2(n), \ldots, D_{s}(n) \right\}.
\]
Our next goal is to use \cite[Lemma~2.14]{abrahamsen2020daugavet} to 
show that if the number of elements $s$
in this set does not grow too fast as $n$ increases,
then $x$ is not a delta-point.

\begin{thm}
  \label{thm:c-infty-finite-no-delta-point}
  Let $X$ be a Banach space with 1-unconditional basis.
  If for $x\in S_X$ there exists $n\in \N$ such
  that for $s = |\bigcup_{D\in \minfx} \left\{D(n) \right\}|$
  \begin{equation*}
    \left\|P_E x \right\| > 1-\frac{1}{2s}
    \quad \mbox{for all} \quad
    E \in \bigcup_{D\in \minfx} \left\{D(n) \right\},
  \end{equation*}
  then $x$ is not a delta-point.

  In particular, if $|\minfx|<\infty$,
  then $x$ is not a delta-point.
\end{thm}

\begin{proof}
  Let $X$ have $1$-unconditional basis $(e_i)_{i \in \N}$ and let $x
  \in S_X$. Since changing signs of the coordinates of vectors is an
  isometry on spaces with 1-unconditional basis
  and delta-points are preserved by isometries
  we may assume that $x = (x_i)_{i \in \N} \in S_X$ with $x_i \ge 0$ for all $i \in \N$.
  By assumption there exists
  $n \in \N$, $s \in \N$ and
  $D_1,\ldots,D_s \in \minfx$ such that
  $\|P_{D_k(n)} x\| > 1 - \frac{1}{2s}$
  for all $k = 1, \ldots, s$.

  For each $k \leq s$ let $x_k^*\in S_{X^*}$
  be such that $x_k^*(P_{D_k(n)}x) =  \left\|P_{D_k(n)}x\right\|$
  and $x_k^*(e_j)=0$ for all $j \notin D_k(n)$.

  Let $x^* = \frac{1}{s}\sum_{k=1}^s {x_k^*}$.
  Then $x^*(x) > 1 - \frac{1}{2s}$, so $x\in S(x^*, \frac{1}{2s})$.
  If $y \in S(x^*, \frac{1}{2s})$ it follows that
  $x_k^*(y) > 1 - \frac{1}{2}$ for all $k = 1, \ldots, s$.
  Fix $k$. If $y_i \le x_i/2$ for all $i \in D_k(n)$,
  then $x_k^*(y) \le \frac{1}{2}x_k^*(x) \le \frac{1}{2}$.
  Thus for each $1 \leq k \leq s$, there exists $i \in D_k(n)$
  such that $y_i > \frac{1}{2}x_i$
  and therefore we can apply
  \cite[Lemma~2.14]{abrahamsen2020daugavet}
  with
  $S(x^*, \frac{1}{2s})$, $n$ and $\eta = \frac{1}{2}$
  and conclude that $x$ is not a delta-point.
\end{proof}

\begin{example}
 Recall that a norm in a Banach lattice $X$ is called \emph{strictly
  monotone} if  $\|x + y\| > \|x\|$ whenever $x,y \in X$ with $x,y\ge
0$ and $y \not= 0.$ From Theorem
\ref{thm:c-infty-finite-no-delta-point} we immediately get that a 
Banach space $X$ with a $1$-unconditional basis and such a
norm has no delta-points since in this case $M(x) = \{\supp(x)\}$
for any $x \in S_X.$
\end{example}

\begin{example}
  The space $X = \ell_1 \oplus_\infty \ell_1$ has a $1$-unconditional
  basis. Moreover, it is easily seen that $|M^\infty(x)| \le 2$ for any $x
  \in S_X,$ so $X$ has no delta-points.
\end{example}

\section{The case $1 < p < \infty$}
\label{sec:case-1-}

In this section we prove the following theorem.

\begin{thm}\label{p_ge_1-ingendelta}
  Let $\mathcal{A}$ be an adequate family of subsets of $\N$
  and let $1 < p < \infty$. Then
  \begin{enumerate}
  \item\label{item:p-1-hap}
    $\hap$ does not have delta-points.
  \item\label{item:p-1-hap-dual}
    $\hap^*$ does not have delta-points.
  \end{enumerate}
\end{thm}

\begin{proof}[Proof of Theorem~\ref{p_ge_1-ingendelta}~\ref{item:p-1-hap}]
  Let $x \in S_{\hap}$. As in the proof of Theorem
  \ref{thm:c-infty-finite-no-delta-point} we assume without loss of
  generality that $x_i \ge 0$ for all $i$.

  As noted above Theorem~\ref{thm:c-infty-finite-no-delta-point}
  we can find
  $k \in \N$ and $D_j \in \minfx$ for $j = 1,\ldots, k$
  such that
  \begin{equation*}
    \bigcup_{D\in \minfx} \{D(1)\} = \{D_1(1),\ldots,D_k(1)\}.
  \end{equation*}
  For each $j = 1,\ldots, k$ define
  \[
    x_j^* = \sum_{i\in D_j} x_i^{p-1}e_{i}.
  \]
  Let $q > 1$ such that $\frac{1}{p} + \frac{1}{q} = 1$.
  Then by Hölders inequality we have
  \begin{align*}
    |x_j^*(y)|
    & \leq
    \sum_{i \in D_j} | x_{i}^{p-1}| |y_{i}|
    \leq
    \left(\sum_{i\in D_j} (|x_{i}^{p-1}|)^{q} \right)^{1/q}
    \left(\sum_{i\in D_j} |y_i|^{p} \right)^{1/p}
    \\
    & \leq
    \|x\|^{p/q} \|y\| \leq 1
  \end{align*}
  As $x_j^*(x) = 1$ we conclude that $x_j^*\in S_{X^*}$.

  For each $j$ write $D_j = (d^j_i)_{i \in \N}$
  with $d^j_i < d^j_{i+1}$.
  Define
  \begin{equation*}
    \xi_j :=
    (x_{d^j_1}^{p-1},
    x_{d^j_2}^{p-1}, \ldots) \in S_{\ell_q}
  \end{equation*}
  and $T_j : \hap \to \ell_p$ by
  \begin{equation*}
    T_j(\sum_{i=1}^\infty a_i e_i)
    =
    (a_{d^j_1}, a_{d^j_2}, \ldots).
  \end{equation*}
  Note that $\|T_j(x)\|_p = 1$ since $D_j \in \minfx$.
  Let
  $\varepsilon := \frac{1}{2} \min_{j} x_{d^j_1} > 0$.
  Clarkson \cite{MR1501880} showed that $\ell_p$
  is uniformly convex so
  there exists $\delta_j > 0$ such that if
  \[
    y \in
    S(\xi_j , \delta_j)
    \subseteq B_{\ell_p}
  \]
  then $\|T_j(x) - y\|_p < \varepsilon$.

  Define $\delta := \min_{j} \delta_j$ and
  $x^* := \frac{1}{k}\sum_{j=1}^k x_j^*$.
  We have $x^*(x) = 1$ and
  if $z\in S( x^*, \frac{\delta}{k})$, then
  \[
    T_j(z) \in S(\xi_j, \delta) \subseteq S(\xi_j, \delta_j),
  \]
  hence
  $|x_{d^j_1} - z_{d^j_1}| \le \|T_j(x) - T_j(z)\| < \varepsilon$.
  By definition of $\varepsilon$ we get
  $z_{d^j_1} \geq \frac{1}{2} x_{d^j_1} > 0$.
  Applying \cite[Lemma~2.14]{abrahamsen2020daugavet}
  with $x^*$, $\delta/k$, $\eta = \frac{1}{2}$ and $n=1$ the result follows.
\end{proof}

For the proof that $\hap^*$ does not
have delta-points we first need to show that the
standard basis is shrinking.

\begin{lem}\label{lem:hap_shrinking}
  Let $\mathcal{A}$ be an adequate family and
  let $1 < p < \infty$.
  Then the standard basis $(e_i)_{i \in \N}$
  of $\hap$ is shrinking.
  In particular, $\hap^*$ is separable
  and has an unconditional basis.
\end{lem}

\begin{proof}
  Let $(x_n)_{n \in \N}$ be a normalized block basis of
  $(e_i)_{i \in \N}$.
  This means that there is a sequence $1 \le p_1 < p_2 < \cdots$
  and coefficients $(a_i^n)$ such that
  $x_n = \sum_{i = p_n}^{p_{n+1} - 1} a_i^n e_i$
  satisfies $\|x_n\| = 1$.

  Define an operator $S : \ell_p \to \hap$ by
  \begin{equation*}
    S((\lambda_n)) = \sum_{n=1}^\infty \lambda_n x_n.
  \end{equation*}
  We have that $S$ is a bounded linear operator.
  Indeed, for $(\lambda_n) \in \ell_p$ and
  $A \in \mathcal{A}$ we define $A_n = A \cap \supp(x_n)$.
  Then we have
  \begin{equation*}
    \sum_{i \in A} |S((\lambda_n))_i|^p
    =
    \sum_{n=1}^\infty \sum_{i \in A_n} |\lambda_n a_i^n|^p
    =
    \sum_{n=1}^\infty |\lambda_n|^p \sum_{i \in A_n} |a_i^n|^p
    \le
    \sum_{n=1}^\infty |\lambda_n|^p \|x_n\|^p
  \end{equation*}
  and hence $\|S\| \le 1$.
  If $(f_n)$ denotes the standard basis in $\ell_p$,
  then $S(f_n) = x_n$.
  Since $S$ is weak--weak continuous
  we get that $(x_n)$ is weakly null since $(f_n)$ is
  weakly null in $\ell_p$.
  By \cite[Proposition~3.2.7]{MR2192298}
  $(e_i)_{i \in \N}$ is shrinking.
\end{proof}

\begin{proof}[Proof of
  Theorem~\ref{p_ge_1-ingendelta}~\ref{item:p-1-hap-dual}]
  By Lemma~\ref{lem:hap_shrinking} the standard basis
  $(e_i)_{i \in \N}$ for $X := \hap$ is shrinking
  and hence the biorthogonal functionals $(e_i^*)_{i \in \N}$
  is an 1-unconditional basis for $X^*$.
  Furthermore, we know that $X^{**}$ is a sequence space
  and that for $x^{**} = (a_i)_{i \in \N} \in X^{**}$
  we have $\|x^{**}\| = \sup_N \|P_N x^{**}\|$
  where $P_N(x^{**}) = \sum_{i=1}^N a_i e_i \in X$.

  Let $x^* = (x_j^*)_{j \in \N} \in S_{X^*}$.
  Without loss of generality we may assume that
  $x_j^* \geq 0$ for each $j$.
  By Theorem~\ref{thm:c-infty-finite-no-delta-point}
  it is enough to show that $M^\infty(x^*)$ is finite.
  To this end we will show that $M(x^*) = \left\{\supp(x^*) \right\}$.

  Assume for contradiction that there exists $j \in \supp(x^*)$ such that
  $\|x^*-x_j^*e_j^* \| = 1$.
  Find $x^{**} \in S_{X^{**}}$ such that
  $x^{**}(x^*-x_j^*e_j^* ) = 1$
  and $x^{**}(e_k^*) = 0$ for all $k \notin \supp(x^*-x_j^*e_j^*)$.
  Let $y^{**} = {x_j^*}^{q-1}e_j + ({1 - {x_k^*}^{q}})^{1/p}x^{**}$.
  For $A \in \mathcal{A}$ and $N \in \N$ denote
  $A_N = A \cap \{1,\ldots,N\}$.
  We have
  \begin{align*}
    \left\|P_N y^{**} \right\|^{p}
    & = \sup_{A\in \mathcal{A}} \sum_{i\in A_N} |y_i^{**}|^p
    \\
    & \leq ({{x_j^*}^{q-1}  })^{p} + \left(1-{x_{j}^*}^{q} \right)
      \sup_{A\in \mathcal{A}} \sum_{i\in A_N}|{x_i^{**}}|^{p}
    \\
    & \leq {x_{j}^{*}}^{q} + \left(1-{x_{j}^{*}}^{q} \right) = 1.
  \end{align*}
  This yields $\|y^{**}\| = \sup_{N}\|P_N y^{**}\| \le 1$
  and since $x^*\in S_{X^*}$
  we arrive at the contradiction
  \[
    y^{**}(x^*) = {x_j^*}^{q}+ \left(1-{x_{j}^*}^{q} \right)^{1/p} > {x_j^*}^{q}+
    \left(1-{x_{j}^*}^{q} \right) = 1.
  \]
  Hence we can leave no index behind and $M(x^*) = \{\supp(x^*)\}$.
\end{proof}

\section{The case $p = 1$}
\label{sec:case-p-er-lik-1}

We now turn to $\ha$ spaces. In this case the situation
is not as clear as for $p > 1$.
Let us first note that the extreme points of the dual
space has a well-known characterization.

\begin{lem}[Lemma~2.3 in \cite{MR1226181}]\label{lem-ext-char-AM}
  Let $\A$ be an adequate family of $\N$.
  Then
  \begin{equation*}
    \ext B_{\ha^*}
    =
    \left\{
      \sum_{i\in A} \varepsilon_i e_i^* :
      A \in \mathcal{A}^{\text{MAX}},
      \varepsilon_i \in \{-1,1\}
    \right\}.
  \end{equation*}
\end{lem}

In \cite[Theorem~3.1]{abrahamsen2020daugavet} it was shown that if we
build an adequate family on $\N$ by using the branches of
the binary tree, then we get an $\ha$ space with delta-points.
This adequate family contains many infinite sets.
We have the following general result about duals of $\ha$ spaces
when the adequate family contains an infinite set.

\begin{prop}\label{prop:infinite-set-delta}
  Let $\mathcal{A}$ be an adequate family that contains
  an infinite set.
  Then there exists $x^*\in \ha^*$
  such that $x^*$ is an extreme point and a delta-point.
\end{prop}

\begin{proof}
  As $\mathcal{A}$ contains an infinite set there exist
  $A \in \mathcal{A}^{\text{MAX}}$ with $|A| = \infty$.
  Write $A = (a_i)_{i=1}^\infty$ with $a_i < a_{i+1}$
  for all $i$.
  For $x^{**} \in S_{\ha^{**}}$ we have
  $y^* =
  \sum_{i=1}^\infty \sgn(x^{**}(e_{a_i}^*)) e_{a_i}^*
  \in B_{\ha^*}$
  by Lemma~\ref{lem-ext-char-AM}.
  Hence
  \begin{equation*}
    \sum_{i=1}^\infty |x^{**}(e_{a_{i}}^*)|
    = x^{**}(y^*) \le 1
  \end{equation*}
  and this implies that $(e_{a_i}^*)_{i=1}^\infty$ is weakly null.

  Now let $x^* = \sum_{i\in A} e_{i}^* \in \ext B_{\ha^*}$.
  Let $x^{**} \in S_{\ha^{**}}$ and $\varepsilon > 0$
  with $x^* \in S(x^{**},\varepsilon)$.
  Define $x_{i}^* = x^* - 2e_{a_i}^* \in \ext B_{\ha^*}$
  so that $\|x^* - x_{i}^*\|  = \|2e_{a_i}^*\| = 2$.
  For $i$ large enough we have $x_i^* \in S(x^{**},\varepsilon)$
  and thus $x$ is a delta-point.
\end{proof}

We shall henceforth exclusively consider $\ha$ spaces
where the adequate family consists of finite sets only.
Recall that a Banach spaces is said to be \emph{polyhedral} if the unit ball
of each of its finite-dimensional subspaces is a polyhedron
\cite{KleePolyhedralSections}.

The following result is well-known but we include
a proof for easy reference.

\begin{prop}\label{prop:ha-poly}
  Let $\A$ be an adequate family of subsets of $\N$.
  The following are equivalent:
  \begin{enumerate}
  \item\label{item:ha-poly-1}
    $\A$ consists of finite sets only;
  \item\label{item:ha-poly-2}
    $\ha$ is polyhedral;
  \item\label{item:ha-poly-3}
    The standard basis $(e_i)_{i \in \N}$ is shrinking.
  \end{enumerate}
\end{prop}

\begin{proof}
  \ref{item:ha-poly-1} $\implies$ \ref{item:ha-poly-2}
  Note that the proof of Theorem 4.5 (see also Remark~4.4) \cite{le2019geometry} 
  also holds for adequate families
  of finite sets, not just for regular families,
  so that $\ha$ is (V)-polyhedral in the sense
  Fonf and Veselý \cite{MR2057283}.
  In particular, $\ha$ is polyhedral.

  \ref{item:ha-poly-2} $\implies$ \ref{item:ha-poly-3}
  If $\ha$ is polyhedral, then it is $c_0$ saturated
  by \cite{MR595744} (see the remark following Theorem~5).
  Hence $\ha$ cannot contain an isomorphic copy of $\ell_1$
  and this yields that
  $(e_i)_{i \in \N}$ is shrinking since it is unconditional.

  \ref{item:ha-poly-3} $\implies$ \ref{item:ha-poly-1}.
  Assume $\mathcal{A}$ contains an infinite set $A$.
  Then the basis vectors $(e_i)_{i \in A}$ span
  an isometric copy of $\ell_1$ in $\ha$ and
  $(e_i)_{i \in \N}$ is not shrinking.
\end{proof}

We will have more to say about polyhedrality in $\ha$ spaces
below in Section~\ref{sec:polyhedrality}.

Next we note that extreme points in $\ha^*$ are actually $w^*$-exposed.

\begin{prop}\label{prop:ext-er-exposed}
  Let $\mathcal{A}$ be an adequate family of $\N$
  and $x^* \in \ext B_{\ha^*}$.
  Then the following are equivalent:
  \begin{enumerate}
  \item\label{item:prop-ext-1}
    $x^*$ is an extreme point of $B_{X^*}$;
  \item\label{item:prop-ext-2}
    $x^*$ is a $w^*$-exposed point of $B_{X^*}$.
  \end{enumerate}
  In particular, if $\mathcal{A}$ is an adequate family of finite
  sets, then $x^*$ is an extreme point
  if and only if $x^*$ is a $w^*$-strongly exposed point.
\end{prop}

\begin{rem}
  We showed above in Proposition~\ref{prop:infinite-set-delta}
  that if the adequate family $\mathcal{A}$ contains
  an infinite set, then there is an extreme point in $\ha^*$
  which is also a delta-point.
  Proposition~\ref{prop:ext-er-exposed} shows that this
  extreme point is also exposed, but being a delta-point
  it is far from being a $w^*$-strongly exposed point.
\end{rem}

\begin{proof}
  One direction is trivial, so we only need to show
  \ref{item:prop-ext-1} $\implies$ \ref{item:prop-ext-2}.
  Let $x^*\in \ext B_{\ha^*}$.
  By Lemma~\ref{lem-ext-char-AM} we can write
  $x^* = \sum_{i\in B} \varepsilon_i e_i^*$ with
  $B \in \mathcal{A}^{\text{MAX}}$
  and $\varepsilon_i \in \{-1,1\}$.
  Find some $y \in S_{\ell_{1}}$  with $\supp(y) = B$ and
  $y_i \geq 0$ for all $i \in \N$.
  Define $x = \sum_{i\in B} y_i\varepsilon_i e_i \in S_{\ha}$
  and notice that $x^*(x) = 1$.

  Take some $y^*\in S_{\ha^*}$ with $y^*(x) = 1$.
  Then $y_i^*=y^*(e_i) = \varepsilon_i$ for all $i \in B$.
  For $j \in \N \setminus B$, there must exist,
  by compactness of $\mathcal{A}$, some finite $A\subset B$
  such that $A \bigcup \left\{j \right\}\notin \mathcal{A}$.
  Define
  \[
    \Tilde x = x+ \sgn
    y_i^*\min_{i\in A}|x_{i}|  e_j,
  \]
  and observe that $\|\Tilde x\| = 1.$ Indeed, take $C \in
  \mathcal{A}$ with $j\in C.$ Then $A\setminus C \not=\emptyset,$ and so
  \[
    \sum_{i\in C} |\Tilde x_i|
     = \sum_{i\in C \cap B} |\Tilde x_i| + |\Tilde x_j|
      \leq \sum_{i\in B} |x_i|- \min_{i\in A} |x_i| + \min_{i \in A}|x_i| = 1.
  \]
  We now get
  \[
    1\ge y^*(\Tilde x )
    =
    y^*(x) + \min_{i\in A}|x_{i}||y_j^*|
    =
    1 + \min_{i\in A}|x_{i}||y_j^*|,
  \]
  implying that $y_j^* = 0$ for all $j\in \N \setminus B$.
  That is, $y^* = x^*$.

  If $\mathcal{A}$ consists of finite sets, then
  $\ha$ is polyhedral and for polyhedral spaces
  the $w^*$-exposed points and $w^*$-strongly exposed
  points of the dual unit ball coincides
  (\cite[Theorem~1.4]{Fonf2000}).
\end{proof}

Let $\mathcal{A}$ be an adequate family and let $x \in S_{\ha}$.
Since $\mathcal{A}$ is compact in the topology of pointwise
convergence in $\N$ there exists $A \in \mathcal{A}$ such
that $\|P_A x\| = \|x\|$. In particular, $\mx \subset \mathcal{A}$.
The following result is now immediate from
Theorem~\ref{thm:c-infty-finite-no-delta-point}
(or \cite[Proposition~2.15]{abrahamsen2020daugavet}).

\begin{prop}
  If $\mathcal{A}$ is an adequate family of finite subsets of $\N$,
  then $\ha$ does not have delta-points.
\end{prop}

\begin{defn}
  A Banach space $X$ has the
  \emph{convex series representation property} (CSRP)
  if for each $x \in B_X$,
  there exists a sequence $(f_i)$ of extreme points of $B_X$
  and a sequence of nonnegative real numbers $(\lambda_i)$
  such that
  $\sum_{i= 1}^{\infty} \lambda_{i} = 1$ and
  \[
    x = \sum_{i=1}^{\infty} \lambda_i f_i.
  \]
\end{defn}
Note that the CSRP is equivalent to the $\lambda$-property
\cite{AronCSRP}.
The proof of \cite[Proposition 4.3]{le2019geometry}
also holds for adequate families of finite sets
and thus we have the following.

\begin{prop}[Proposition~4.3 in \cite{le2019geometry}]
  \label{prop:ha_characterize_ext_andlambda_property}
  Let $\A$ be an adequate family of finite subsets of $\N$.
  Then $\ha^*$ has the CSRP.
\end{prop}

Using this proposition we are able to show that
for regular families $\mathcal{A}$ the dual of $\ha$
does not have delta-points.

\begin{prop}
  If $\mathcal{A}$ is an adequate family of finite subsets of $\N$
  which is spreading, then $\ha^*$ fails to have any delta-points.
\end{prop}

\begin{proof}
  By assumption the standard basis $(e_i)_{i \in \N}$ for $\ha$ is
  shrinking and hence the biorthogonal functionals
  $(e_i^*)_{i \in \N}$ is a 1-unconditional basis for $\ha^*$.

  We only need to consider $x^*$ with infinite support
  by Theorem~\ref{thm:c-infty-finite-no-delta-point}.
  From Lemma~\ref{lem-ext-char-AM} we have that
  elements of $\ext \ha^*$ can be written
  \begin{equation*}
    x^*_{A,(\varepsilon_i)} = \sum_{i \in A} \varepsilon_i e_i^*
    \quad \mbox{where} \quad
    A \in \mathcal{A}^{\text{MAX}}\ \mbox{and}\ (\varepsilon_i) \in \{-1,1\}.
  \end{equation*}
  By Proposition~\ref{prop:ha_characterize_ext_andlambda_property}
  we can write
  $x^* = \sum_{n =1}^\infty \lambda _n x^*_{F_n,(\eps_i^n)}$
  where for $F_n \in \mathcal{A}$ and $\varepsilon_i^n \in \{-1,1\}$.
  There are two possibilities. Either $\|x^* - x_k^* e^*_k\| < 1$,
  for every $k \in \supp(x^*)$, and then
  $M^{\infty}(x^*) = \{\supp(x^*)\}$,
  and the result follows from
  Theorem~\ref{thm:c-infty-finite-no-delta-point}.
  The other possibility is that
  there exists $k \in \supp(x^*)$ such that
  $\|x^* - x_k^* e^*_k\| = 1$. Write
  \[
    y^* := x^* - x_k^* e^*_k = \sum_{n=1}^\infty
    \lambda_n x_{F_n',(\eps_i^n)}^*,
  \]
  where $F_n' = F_n \setminus \{k\}$ for all $n$.
  Find
  $y^{**} = (y_i) \in S_{X^{**}}$ such that
  \[
    1
    = y^{**}(y^*)
    = \sum_{n=1}^\infty \lambda_n y^{**}(x_{F_n', (\eps_i^n)}^*).
  \]
  This implies that $y^{**}(x_{F_n', (\eps_i^n)}^*) = 1$
  for every $n \in \N$.
  Now find $n \in \N$ with $k \in F_n$.
  Take any integer $j \in (F_n')^C$ with
  $j \ge \max F_n$, let $G_n := F_n' \cup \{j\}$,
  and let $\eps_j = \sgn y_j$.
  As $\mathcal{A}$ is spreading $G_n \in \mathcal{A}$ and thus
  $\|x^*_{G_n, (\eps_i^n)}\| = 1$.
  But then,
  \[
    1
    \ge y^{**}(x^*_{G_n, (\eps_i^n)})
    = y^{**}(x_{F_n', (\eps_i^n)}^*) + y^{**}(\eps_je^*_j)
    = 1 + |y_j| \ge 1,
  \]
  which forces $y_j$ to be zero.
  From this it follows that $y^{**} = (y_i)$ is zero
  on all coordinates $i \ge \max F_n$ (at least).
  That is, every $A \in M(x^*)$ must be a subset of
  $\left\{1, \ldots, \max F_n \right\}$, so
  $M^{\infty}(x^*) = \emptyset$ and the result follows from
  Theorem~\ref{thm:c-infty-finite-no-delta-point}.
\end{proof}

If $\mathcal{A}$ is an adequate family of finite sets
which is not spreading, we do not know whether or not
$\ha^*$ can have delta-points. But we are able to say
something about Daugavet-points.

Let $X$ be a Banach space.
Recall that $x \in S_X$ is a Daugavet-points
if and only if for every slice $S(x^*,\delta)$ of $B_X$ and
for every $\varepsilon > 0$ there exists $y \in S(x^*,\delta)$
such that $\|x - y\| \geq 2-\varepsilon$
(see e.g. \cite[Lemma~2.3]{AHLP}).

\begin{prop}
  If $x^* \in S_{X^*}$ can be written as
  \begin{equation*}
    x^* = \sum_{i=1}^\infty \lambda_i f_i
  \end{equation*}
  where $\lambda_i \ge 0$, $\sum_{i=1}^\infty \lambda_i = 1$,
  $f_i$ are ($w^*$-)strongly exposed points
  in $B_{X^*}$, then $x^*$ is not a Daugavet-point.

  In particular,
  if $\mathcal{A}$ is an adequate family of finite sets,
  then $\ha^*$ does not have Daugavet-points.
\end{prop}

\begin{proof}
  Choose $j$ with $\lambda_j > 0$.
  Find a slice $S(x^{**},\delta)$ containing $f_j$ with diameter less than $1$.
  Then for any $y^* \in S(x^{**},\delta)$ we have
  \begin{align*}
    \|x^* - y^*\|
    &\le
    \|\sum_{i \neq j} \lambda_i f_i\|
    +
    \|y^* - \lambda_j f_j\|\\
    &\le
    \sum_{i \neq j} \lambda_i
    + (1-\lambda_j)
    + \lambda_j\|y^* - f_j\|
    \\
    &< 2(1-\lambda_j) + \lambda_j
    = 2 - \lambda_j.
  \end{align*}
  So the distance from $x^*$ to $y^*$ is bounded away from $2$.
\end{proof}

\section{Polyhedrality}
\label{sec:polyhedrality}

In this section we study polyhedrality in $\ha$ spaces.
The goal is to describe the polyhedrality of $\ha$
spaces in terms of the structure of the adequate family $\mathcal{A}$.
Let us begin by recalling some concepts and results.

If $X$ is a Banach space and  $A \in X^*$ then $A'$ denotes
the set of all $w^*$\emph{-limit points} of $A$,
that is
\[
  A' = \left\{
    f \in X^* \
    :
    \  f \in  \overline{A\setminus \{f\} }^{w^*}
  \right\}.
\]

In \cite{MR2057283} Fonf and Veselý identified eight known
definitions of polyhedrality from the literature and
gave examples showing that in general they are different.
We will use the following three definitions from their paper.
\begin{defn}\label{defn:polyhedrlaity}
  Let $X$ be a Banach space. Then
  \begin{enumerate}
  \item
    $X$ is (I)-polyhedral if
    $\left(\ext B_{X^*} \right)' \subseteq \{0\}$;
  \item
    $X$ is (IV)-polyhedral if $f(x) < 1$ whenever
    $x\in S_{X}$ and $f \in \left(\ext B_{X^*} \right)'$;
  \item
    $X$ is (V)-polyhedral if
    $\sup\left\{
      f(x) \ : \  f \in \ext B_{X^*} \setminus D(x)
    \right\} < 1$
    for each $x\in S_X$,
    where $D(x) = \left\{f \in S_{X^*} : f(x) = 1 \right\}$.
  \end{enumerate}
\end{defn}
In the proof of Proposition~\ref{prop:ha-poly} we met the following:

\begin{thm}[Theorem~4.5 in \cite{le2019geometry}]
  \label{thm:ha-vpol}
  Let $\A$ be an adequate family of finite subsets of $\N$.
  Then $\ha$ is (V)-polyhedral.
\end{thm}

It is natural to ask whether any (V)-polyhedral $\ha$ can be
(IV)- or even (I)-polyhedral.
Our next goal is to show the following:
If $\ha$ is polyhedral, then it is either
(I)-polyhedral or (V)-polyhedral.
Considering Proposition~\ref{prop:ha-poly} and
Theorem~\ref{thm:ha-vpol} we only need to show
the following.

\begin{thm}\label{thm:polyhedrality-I-IV}
  Let $\mathcal{A}$ be an adequate family of finite sets.
  Then the following are equivalent:
  \begin{enumerate}
  \item\label{item:ph-ch-1}
    $\left\{A\in \mathcal{A}  \ : \  i\in A \right\}$
    is finite for all $i \in \N$.
  \item\label{item:ph-ch-2}
    $\ha$ is (I)-polyhedral.
  \item\label{item:ph-ch-3}
    $\ha$ is (IV)-polyhedral.
  \end{enumerate}
\end{thm}

\begin{proof}
  \ref{item:ph-ch-2} $\implies$ \ref{item:ph-ch-3}
  is trivial.

  \ref{item:ph-ch-1} $\implies$ \ref{item:ph-ch-2}.
  By assumption and Lemma~\ref{lem-ext-char-AM},
  for any $i\in \N$, there are only a finite number of extreme points
  that have support on $i$. It follows that if
  $f\in (\ext B_{X^*})'$, then $f(e_i) = 0$ for all
  $i\in \N$. Therefore $f$ is $0$.

  \ref{item:ph-ch-3} $\implies$ \ref{item:ph-ch-1}.
  First observe that if
  $\left\{A\in \mathcal{A} \ : \ i \in A  \right\}$
  is infinite, then the set
  $\mathcal{C}_i  =
  \left\{A\in \mathcal{A}^{\text{MAX }}   \ : \ i\in A \right\}$
  is also infinite.

  Assume that for $i \in \N$ the set $\mathcal{C}_i$ is infinite
  and let $(A_n)_{n=1}^\infty \subset \mathcal{C}_i$ be a sequence of
  distinct sets.
  As elements of $\mathcal{A}$ are finite and
  $\mathcal{A}$ is compact
  in the topology of pointwise convergence on $\N$,
  we can by passing to a subsequence if necessary assume that
  $(A_n)$ converges pointwise to some $A\in \mathcal{A}$
  and that $A\subset A_n$ for all $n\in \N$.

  If $k \in \N$ then there exists $N \in \N$
  such that $\left\{1, 2, \ldots, k \right\} \cap A_n
  = \left\{1, 2, \ldots, k \right\}\cap A$ for all $n>N$,
  i.e. $\min (A_n\setminus A) \rightarrow \infty$.

  With the sequence $(A_n)$, we wish to show that $\ha$ is not
  (IV)-polyhedral.
  That is, we wish to show that there exist some
  $x^* \in (\ext B_{\ha^*})'$ and some $x \in S_{\ha}$ such that $x^*(x) = 1$.
  To this end, define $x^* = \sum_{i\in A} e_{i}^*$ and
  $x = \frac{1}{\left|A \right| } \sum_{i\in A} e_i\in S_{\ha}$.
  As $x^*(x) = 1$, it remains to show that $x^*$ is a $w^*$-limit
  of elements in $\ext B_{\ha^*}\setminus \{x^*\}$.

  By Lemma~\ref{lem-ext-char-AM}
  $x_n^* := \sum_{i\in A_n} e_i^*
  = x^* + \sum_{i\in A_n \setminus A}e_i^*  \in \ext B_{\ha^*}$.
  If
  $y \in \ha$  we get that
  \begin{align*}
    |x_n^*(y)  -x^*(y)|
    &=
    | x^*(y) + \sum_{i\in A_n\setminus A} y_i  - x^*(y)|
    \\
    &\leq
    \sum_{i \in A_n \setminus A} |y_i|.
  \end{align*}
  Since $\sum_{i\in A_n\setminus A} |y_i| \rightarrow 0$
  because $\min (A_n \setminus A) \rightarrow \infty$
  we get $x_n^* \xrightarrow{w^*} x^*$
  and thus $\ha$ is not (IV)-polyhedral.
\end{proof}

The above theorem shows that the only combinatorical Banach space
which is (IV)-polyhedral is $c_0$.
Antunes, Beanland and Chu \cite[Theorem~4.5]{le2019geometry}
(see also their Remark~4.4)
proved that combinatorial Banach spaces are (V)-polyhedral.
So our next corollary shows that this is in fact best possible.

\begin{cor}
  If $\mathcal{A}$ is a regular family of sets
  such that $\ha$ is (IV)-polyhedral,
  then $\ha = c_0$.
\end{cor}

\begin{proof}
  Assume that $A \in \mathcal{A}$ with $|A| > 1$.
  Let $i = \min A$. Since there are infinitely many
  spreads of $A$ we have that
  $\left\{A\in \mathcal{A}  \ : \  i\in A \right\}$
  is infinite.
\end{proof}

We end the paper with some questions.

Let $\mathcal{A}$ be an adequate family.
Proposition~\ref{prop:ha-poly} gives a simple
characterization of when $\ha$ contains a copy of $\ell_1$
and $\hap$ never contains $\ell_1$ by Lemma~\ref{lem:hap_shrinking}
(and James' characterization of shrinking unconditional bases).

Note that $\ha$ and even $\hap$ contains an isometric
copy of $c_0$ if there exists an infinite subset
$E\subseteq \N$ such that $|A \cap E| \le 1$
for all $A \in \mathcal{A}$. We ask:

\begin{quest}
  What is a natural condition on $\mathcal{A}$
  such that $\ha$ (or $\hap$) does not contain $c_0$.
  That is, when does $\hap$ have a boundedly complete basis?
\end{quest}

We have seen that if $\mathcal{A}$ is an adequate
family of finite sets which is spreading, then
neither $\ha$ nor $\ha^*$ have delta-points.
We do not know if that is also the case
if $\mathcal{A}$ is not spreading.

\begin{quest}
  If $\ha$ is (V)-polyhedral, 
  does then $\ha^*$ fail to have delta-points?

  One can even ask: If $\ha$ is (I)-polyhedral,
  does then $\ha^*$ fail to have delta-points?
  Note that if $X$ is (I)-polyhedral,
  then $X$ is isometric to a subspace of $c_0$
  \cite[Theorem~1.2]{MR2057283}.
\end{quest}

\section*{Acknowledgements}
The authors would like to thank Stanimir Troyanski for discussions on
the topic of the paper.

\def\cprime{$'$}
\providecommand{\bysame}{\leavevmode\hbox to3em{\hrulefill}\thinspace}
\providecommand{\MR}{\relax\ifhmode\unskip\space\fi MR }
\providecommand{\MRhref}[2]{%
  \href{http://www.ams.org/mathscinet-getitem?mr=#1}{#2}
}
\providecommand{\href}[2]{#2}

\end{document}